\documentclass{scrartcl}
\usepackage{amsmath, amssymb, amsthm}
\usepackage{graphicx}
\usepackage{enumerate}

\title{Hypergraph Saturation Irregularities}
\author{Natalie C. Behague, Queen Mary University of London}

\theoremstyle{plain}
\newtheorem{theorem}{Theorem}

\newtheorem{conjecture}[theorem]{Conjecture}
\newtheorem{question}[theorem]{Question}

\theoremstyle{plain}
\newtheorem*{theorem*}{Theorem}

\theoremstyle{remark}
\newtheorem{claim}{Claim}[theorem]

\newcommand{\family}{\mathcal{F}} 
\newcommand{\sat}{\operatorname{sat}(\mathcal{F},n)} 

\newcommand{\repeatthm}[2]{
\theoremstyle{plain}
\newtheorem*{#1}{Theorem \ref{thm:#1}}
\begin{#1}
#2
\end{#1}
}

\DeclareMathOperator{\edge}{e}

\begin{document}
\maketitle
\begin{abstract}
Let $\mathcal{F}$ be a family of $r$-graphs. An $r$-graph $G$ is called $\mathcal{F}$-saturated if it does not contain any members of $\mathcal{F}$ but adding any edge creates a copy of some $r$-graph in $\mathcal{F}$. The saturation number $\operatorname{sat}(\mathcal{F},n)$ is the minimum number of edges in an $\mathcal{F}$-saturated graph on $n$ vertices. 
We prove that there exists a finite family $\mathcal{F}$ such that $\operatorname{sat}(\mathcal{F},n) / n^{r-1}$  does not tend to a limit. This settles a question of Pikhurko.
\end{abstract}

\section{Introduction}

An \emph{$r$-graph} is a pair, $(V(H),E(H))$, of \emph{vertices} and \emph{edges} where the edge set $E(H)$ is a collection of $r$-element subsets of the vertex set $V(H )$. We will let $|H| = |V(H)|$ and $e(H) = |E(H)|$.
When the context is clear we will refer to $r$-graphs simply as graphs.

Fix an $r$-graph $F$, called a forbidden graph. An $r$-graph $H$ is said to be  \emph{$F$-saturated} if $H$ does not contain $F$ as a subgraph, but adding any extra edge to $H$ creates a copy $F$ as a subgraph of $H$. We define the saturation number
\[\operatorname{sat}(F,n) = min\{e(H): |H| = n \text{ and $H$ is $F$-saturated}\}.\]

For $2$-graphs, (i.e. graphs), K\'aszonyi and Tuza \cite{KT86} proved that ${\operatorname{sat}(F,n) = O(n)}$. As a result, Tuza \cite{Tuz88} conjectured the following:

\begin{conjecture}[Tuza] For every $2$-graph $F$ the limit $\lim_{n\rightarrow \infty}\frac{\operatorname{sat}(F,n)}{n}$ exists.
\label{conjecture}
\end{conjecture}

We can generalise the notion of saturation to families of graphs. For a family $\family$ of $r$-graphs  (called a forbidden family), an $r$-graph $H$ is called \emph{$\family$-saturated} if it does not contain any graph in $\family$ as a subgraph, but adding any edge creates a copy of some graph $F \in \family$ as a subgraph of $H$. We define the saturation number in the same way as before:
\[\operatorname{sat}(\family,n) = min\{e(H): |H| = n \text{ and $H$ is $\family$-saturated}\}.\]

For a family $\family$ of $2$-graphs we have ${\sat = O(n)}$ \cite{KT86}, just as we did for single $2$-graphs. However, the generalisation of Tuza's conjecture to finite families of graphs is not true: an example of a finite family  $\family$ where $\sat/n$ does not tend to a limit was given by Pikhurko \cite{Pik04}.
In this example, the graphs in $\family$ depend on some fixed constant $k$. For $n$ divisible by $k$, one can construct an $\family$-saturated graph on $n$ vertices that uses relatively few edges. For $n$ not divisible by $k$, there is no such `nice' construction and an $\family$-saturated graph on $n$ vertices is forced to contain many extra edges. 

For a family $\family$ of $r$-graphs, it was shown by Pikhurko \cite{Pik99} that  $\sat = O\left(n^{r-1}\right)$ when the family contains only a finite number of graphs.
This leads to the following generalisation of Tuza's conjecture to $r$-graphs, first posed by Pikhurko \cite{Pik99}.

\begin{conjecture} For every r-graph $F$ the limit $\lim_{n\rightarrow \infty}\frac{\operatorname{Sat}(F,n)}{n^{r-1}}$ exists.
\label{gen_conjecture}
\end{conjecture}

As in the $2$-graph case we can further generalise this conjecture by  replacing the single $r$-graph $F$ with a finite family of $r$-graphs $\family$. Our main aim in this paper is to prove that this generalised conjecture is not true -- that is, for all $r$ there exists a finite family of $r$-graphs $\family$ such that $\sat / n^{r-1}$ does not tend to a limit. This resolves a question of Pikhurko (problem 7 in \cite{Pik04}). 

\repeatthm{big family}{
For all $r\ge 2$ there exists a family $\family$ of $r$-graphs and a constant $k \in \mathbb{N}$ such that 
	$$ \sat =
		\begin{cases}
			O(n) & \text{if } k \mid n \\			
			\Omega\left(n^{r-1}\right) & \text{if } k \nmid n	
		\end{cases} 
	$$
In particular, for any ${l \in \{1,2,\ldots,r-1\}}$,  we have that	
$\frac{\sat}{n^l}$ does not converge.
}

We prove Theorem \ref{thm:big family} in Section \ref{sec:main thm}. As in Pikhurko's proof for the $2$-graph case, the idea of the proof will be to choose a constant $k$ and to define a forbidden family $\family$ such that when $k$ divides $n$ there is a `nice' construction of an $\family$-saturated graph with few edges; and when $k$ does not divide $n$, an $\family$-saturated graph requires comparatively many edges. 

Our proof of Theorem \ref{thm:big family} uses a family $\family$ which grows in size with $r$. In a variation of the theorem, proved in Section \ref{sec:constant size}, we show that we can reduce the size of the forbidden family to be independent of $r$.

\repeatthm{constant size family}{
For all $r \ge 3$ there exists a family $\family$ of four $r$-graphs such that $\frac{\sat}{n^{r-1}}$ does not converge.
}

In reducing the family to a constant size we lose the large gap between the asymptotics that we had in Theorem \ref{thm:big family}. In particular, for a choice of constant $k$, we still have that if $k \nmid n$ then $\sat = \Omega\left(n^{r-1}\right)$, but if $k | n$ we only have $\sat = O\left(n^{r-2}\right)$ (as opposed to the $O(n)$ we had before).

Consider, with respect to the family given in Theorem \ref{thm:big family}, the set of integers $n$ where $\sat$ is $O(n)$. This set has low density: specifically, density $1/k$ where $k$ grows with $r$. A second variation of the theorem gives a forbidden family such that the the set of integers $n$ where $\sat$ is $O(n)$ has density $1/2$. This is proved in Section \ref{sec:dense}.

\repeatthm{evens and odds}{
For all $r \ge 2$ there exists a family $\family$ of $r$-graphs such that 
	$$ \sat =
		\begin{cases}
			O(n) & \text{if $n$ is even}\\			
			\Omega\left(n^{r-1}\right) & \text{if $n$ is odd.} 
		\end{cases} 
	$$
}

We end the paper with some open problems in Section \ref{sec:problems}.

\section{A Proof of the Main Theorem}
\label{sec:main thm}

\begin{theorem}  \label{thm:big family}
For all $r\ge 2$ there exists a family $\family$ of $r$-graphs and a constant $k \in \mathbb{N}$ such that 
	$$ \sat =
		\begin{cases}
			O(n) & \text{if } k \mid n \\			
			\Omega\left(n^{r-1}\right) & \text{if } k \nmid n	
		\end{cases} 
	$$
In particular, for any ${l \in \{1,2,\ldots,r-1\}}$,  we have that	
$\frac{\sat}{n^l}$ does not converge.
\end{theorem}

\begin{proof}

Fix any integer $k>r$ and take $\family$ to be the family of all of the following $r$-graphs:
\begin{enumerate}[a)]
\item For each $1 \le i \le k-1$, the graph $F_i$ consisting of two copies of $K_k^{(r)}$ intersecting in exactly $i$ vertices.
\item For each $(x_1,x_2,\ldots, x_t)$ with  $\Sigma x_i = r$ and ${1 \le x_1 \le x_2 \le \ldots \le x_t \le (r-1)}$, the graph $H_{(x_1,x_2,\ldots, x_t)}$ consisting of $t$ disjoint copies of $K_k^{(r)}$ and an edge $E$ meeting the $i^{th}$ copy of $K_k^{(r)}$ in $x_i$ vertices. We refer to $E$ as the bridge edge.
\end{enumerate}

An example of the family $\family$ for $r=5$ and $k=7$ can be seen in figure \ref{fig:big_family}, where the vertices surrounded by a dashed line represent a copy of $K_k^{(r)}$, and vertices grouped by a solid line represent a bridge edge.
\begin{figure}[!h]
  \centering
    \includegraphics[width=\textwidth]{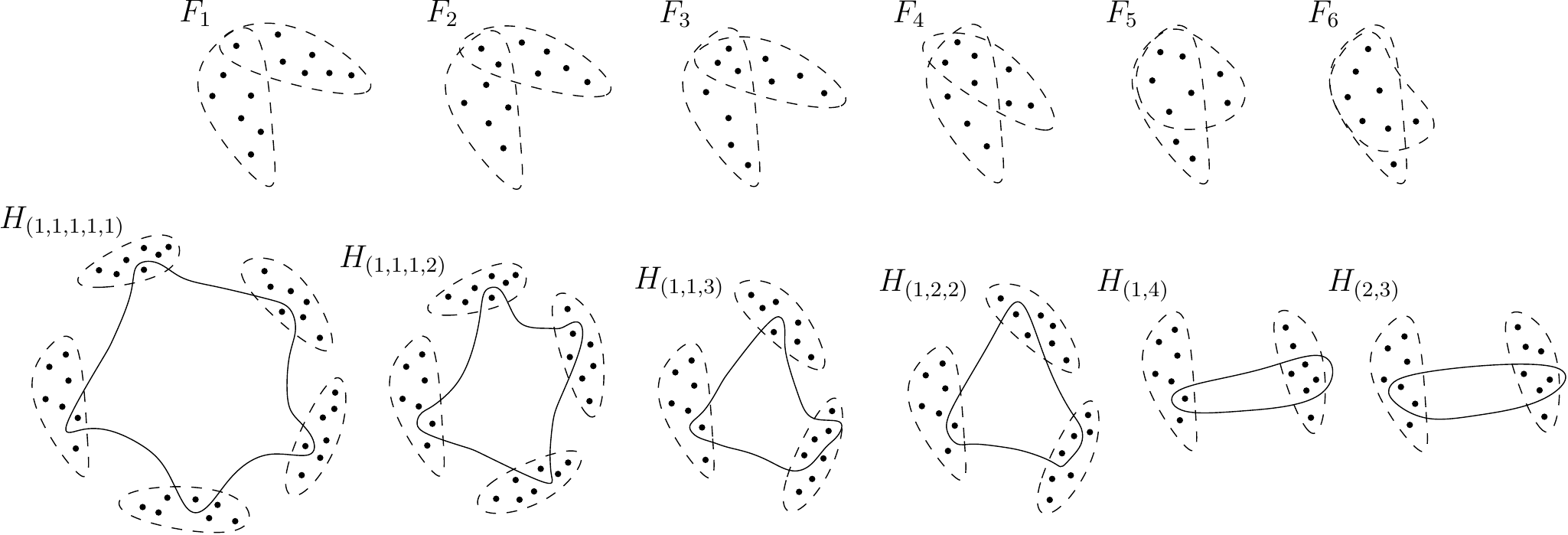}
    \caption{The family $\mathcal{F}$ of $r$-graphs for $r=5$ and $k=7$}
    \label{fig:big_family}
\end{figure}

First, let us deal with the case where $k$ divides $n$. Take the graph consisting of $\frac{n}{k}$ disjoint copies of $K_k^{(r)}$. It is easy to see that this is $\family$-saturated and thus \[\sat \le \frac{n}{k} \binom{k}{r} = O(n).\]

Note that in fact $\sat$ is equal to $\frac{n}{k} \binom{k}{r}$ (although we do not require this).

Now suppose that $k \nmid n$ and let $G = (V,E)$ be a graph on $n$ vertices that is $\family$-saturated. 
We will show that $E(G) = \Omega\left( n^{r-1} \right)$.
Let $A$ be the set of all vertices of $G$ that are contained in a $K_k^{(r)}$, and $B = V\setminus A$ be all vertices not contained in any $K_k^{(r)}$.

\begin{figure}[!h]
  \centering
    \includegraphics[width=.7\textwidth]{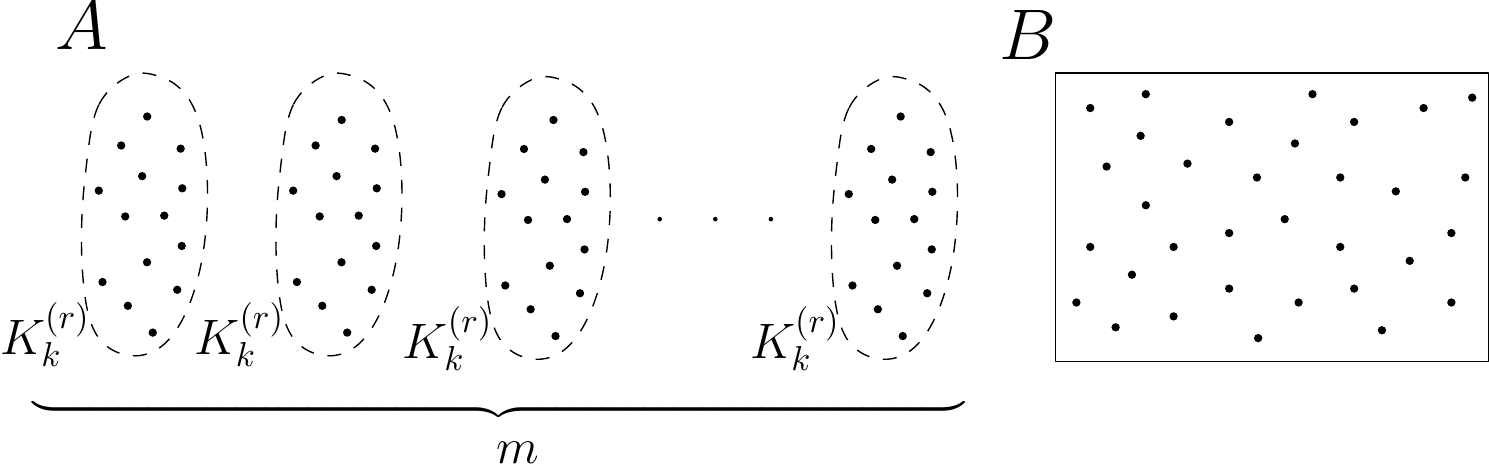}
    \caption{The structure of the graph $G$}
    \label{fig:A_and_B}
\end{figure}

Note that the subgraph of $G$ induced by $A$ must consist of $m$ disjoint copies of $K_k^{(r)}$ for some $m \ge 0$, by the choice of family $\family$. This implies that $B$ is not empty, since $k$ does not divide $n$.
Note also that if an $r$-set intersecting $B$ is not in $E(G)$, then adding this edge to $G$ must create a copy of $K_k^{(r)}$ -- it must create some graph in $\family$ and it cannot form a bridge between two or more $K_k^{(r)}$s by definition of $B$.

We make the following two claims about the number of edges in $G$:

\begin{claim} $G$ contains at least $\binom{mk}{r-1} - m\binom{k}{r-1}$ edges consisting of $r-1$ vertices in $A$ and one vertex in $B$.
 \label{claim1:edges between A and B}
\end{claim}

\begin{claim}
$G$ contains at least $\binom{|B|}{r-1} \frac{k-r}{r}$ edges.
\label{claim1:degrees in B}
\end{claim}

We can use these two claims to deduce the result. One of $A$ and $B$ contains at least half the vertices in $G$. If $|A| = mk \ge \frac{n}{2}$, then using claim \ref{claim1:edges between A and B} the number of edges in $G$ is $\Omega\left(n^{r-1}\right)$. If $|B| \ge \frac{n}{2}$, then using claim \ref{claim1:degrees in B} the number of edges in $G$ is $\Omega \left(n^{r-1}\right)$.

All that is left is to prove the two claims.

\begin{proof}[Proof of Claim \ref{claim1:edges between A and B}.]
This holds trivially if $m$ is $0$ or $1$, so we may assume $m \ge 2$.
Fix a set $X$ of $r-1$ vertices in $A$, not all in the same copy of $K_k^{(r)}$. 
We will show that $G$ contains at least one edge containing all vertices of $X$ together with a vertex in $B$. This proves the claim since there are  $\binom{mk}{r-1} - m\binom{k}{r-1}$ such sets $X$.

Fix some $x$ in $B$ (note that $B$ is non-empty) and suppose that the $r$-set $X \cup \{x\}$ is not in $E(G)$.
Since $G$ is $\family$-saturated, adding $X \cup \{x\}$ as an edge must create a copy of $K_k^{(r)}$ on some vertex set $K$. 

Note that for $y$ in $A \setminus X$, the $r$-set $X \cup \{y\}$ is not in $E(G)$, as otherwise it would form a copy of some  $H_{(x_1,\ldots,x_t)}$ in $\family$. Thus the vertices in $K \setminus ( X \cup \{x\})$ cannot be in $A$. 

Hence $K \setminus ( X \cup \{x\})$ is contained entirely in $B$, and so $G$ contains $k-r > 1$ edges that consist of all vertices of $X$ together with a vertex in $B$.
\end{proof}

\begin{proof}[Proof of Claim \ref{claim1:degrees in B}.]
Fix a set $X$ of $r-1$ vertices in $B$. We will show that $X$ is contained in at least $k-r$ edges of $G$.

Suppose first that we have that $X \cup \{y\}$ is an edge for all $y$ in $B \setminus X$. Then $X$ is in $|B| - |X| \ge k-r$ edges as required.

Otherwise, there exists some $y$ in $B\setminus X$ such that $X \cup \{y\}$ is not in $E(G)$. Then adding the edge $X \cup \{y\}$ must create a copy of $K_k^{(r)}$ in $G$, since $B$ contains no $K_k^{(r)}$s and so this edge cannot be a bridge edge. Then we have that $X$ is contained in $k-r$ other edges in that $K_k^{(r)}$.

Thus every $(r-1)$-set in $B$ is contained in at least $k-r$ edges. Each edge in $G$ contains at most $r$ different $(r-1)$-sets in $B$, and so the total number of edges in $G$ is at least
\[ \binom{|B|}{r-1} \frac{k-r}{r}. \]
\end{proof}

\end{proof}

Note that the size of the family $\family$ used in this proof is $(r - 1) + (p(r) - 1)$, where $p$ is the partition function. We have $\log{(p(r))} = \Theta(\sqrt{r})$, and so the size of the family grows exponentially with $\sqrt{r}$.

\section{A Forbidden Family of Constant Size}
\label{sec:constant size}

Recall that Tuza's conjecture (and its generalisation to $r$-graphs) concerns a forbidden family of just one $r$-graph. The size of the forbidden family in Theorem \ref{thm:big family} grows with $r$, the size of each edge. It is natural then to ask whether there exists a family of constant size, independent of $r$, which has this same non-convergence property.

We will prove that there is such a  family, using the family $\family$ consisting of the following four $r$-graphs:
\begin{itemize}
\item[$F$:] Two $K_k^{(r)}$s intersecting in one vertex,
\item[$H$:] $r$ disjoint copies of $K_k^{(r)}$ and an edge intersecting each $K_k^{(r)}$,
\item[$I_2$:] One $K_k^{(r)}$ and an edge intersecting it in exactly two vertices, and
\item[$I_{r-1}$:] One $K_k^{(r)}$ and an edge intersecting it in exactly $r-1$ vertices.
\end{itemize}

An example of the family $\family$ for $r=5$ and $k=15$ can be seen in figure \ref{fig:constant_family}, where the vertices surrounded by a dashed line represent a copy of $K_k^{(r)}$, and vertices grouped by a solid line represent an extra edge.

\begin{figure}[!h]
  \centering
    \includegraphics[width=.8\textwidth]{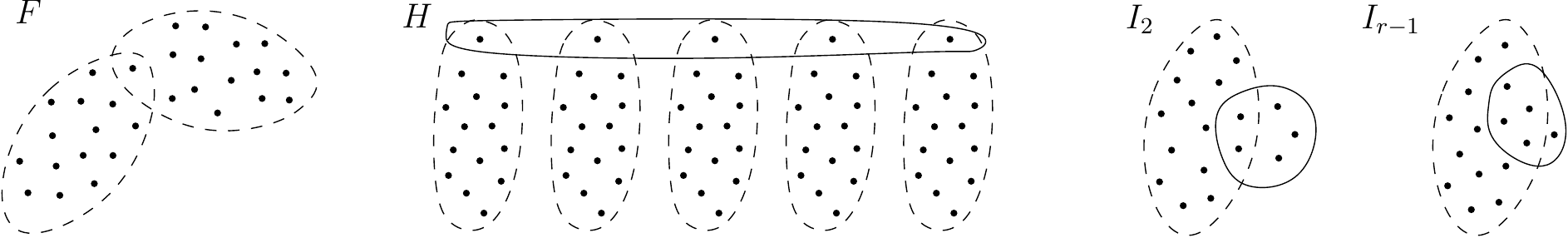}
    \caption{The family $\mathcal{F}$ of $r$-graphs for $r=5$ and $k=15$}
    \label{fig:constant_family}
\end{figure}

This family was obtained by considering the two types of graph we had in our previous family, and finding a smaller set of graphs that fulfil the same role for each. 

The previous family contained all of the graphs $H_{(x_1,\ldots,x_t)}$ to ensure that the graph consisting of disjoint copies of $K_k^{(r)}$ was $\family$-saturated. In particular, this meant that when $k \mid n$ there was an $\family$-saturated graph of size $O(n)$. We will keep ${H = H_{(1,1,\ldots,1)}}$ to ensure there are no edges intersecting $r$ different copies of $K_k^{(r)}$, and replace the other $p(r) - 2$ graphs by $I_2$ and $I_{r-1}$. With this smaller family we can no longer find a graph of size $O(n)$ that is $\family$-saturated, so we lose the large gap in asymptotics that we had in Theorem \ref{thm:big family}. However, we can construct (for $n$ divisible by $k$) an $\family$-saturated graph that has size $O\left( n ^{r-2} \right)$.

The previous family contained all of the graphs  $F_i$ to ensure that all of the copies of $K_k^{(r)}$ in an $\family$-saturated graph must be disjoint. For $k$ sufficiently large, these $r-1$ different graphs can be replaced by the three graphs $I_2$, $I_{r-1}$ and $F = F_1$, which achieve the same goal.

\begin{theorem} \label{thm:constant size family}
For all $r \ge 3$ there exists a family $\family$ of four $r$-graphs such that $\frac{\sat}{n^{r-1}}$ does not converge.
\end{theorem}

\begin{proof}[Proof of Theorem \ref{thm:constant size family}.]
Fix $r\ge3$ and $k \ge \max\{r+1, 2r-4\}$.

Let $\family$ be the set containing the four $r$-graphs $F, H, I_2$ and $I_{r-1}$, as defined earlier.

First, we will construct an example of an $\family$-saturated graph with $O\left(n^{r-1}\right)$ edges when $k$ divides $n$. 
\begin{claim} $\sat = O\left(n^{r-2}\right)$ when $k\mid n$. \label{claim2:when k|n}
\end{claim}
\begin{proof}
Let $G$ be a graph consisting of $\frac{n}{k} = m $ disjoint copies of $K_k^{(r)}$, together with all other edges \emph{except} those which:
\begin{enumerate}[i)]
\item \label{item:intersect_1} intersect $r$ of the $K_k^{(r)}$s, each in one vertex, or
\item \label{item:intersect_2} intersect one of the $K_k^{(r)}$s in exactly two vertices, or
\item \label{item:intersect_r-1}intersect one of the $K_k^{(r)}$s in exactly $r-1$ vertices.
\end{enumerate}

Clearly, adding any $r$-set not in $E(G)$ to the graph $G$ creates one of the graphs in the family $\family$ -- graphs $I_2$, $I_{r-1}$ and $H$ respectively.

We will show that $G$ contains no other $K_k^{(r)}$s except for the original ones, thus proving that $G$ does not contain any graph in $\family$. 
Suppose $G$ did contain another $K_k^{(r)}$. Using that $k \ge 2r-2$, we have that one of the following three cases holds:
\begin{itemize}
\item the new $K_k^{(r)}$ intersects all of the original $K_k^{(r)}$s in at most one vertex, and thus it contains an edge of type (\ref{item:intersect_1});
\item the new $K_k^{(r)}$ and one of the original $K_k^{(r)}$s intersect in between two and $k-(r-2)$ vertices, and thus it contains an edge of type (\ref{item:intersect_2}); or
\item the new $K_k^{(r)}$ and one of the original $K_k^{(r)}$s intersect in $r-1$ or more vertices, and thus it contains an edge of type (\ref{item:intersect_r-1}).
\end{itemize}
Whichever case we are in, we have a contradiction. Thus $G$ is $\family$-saturated. 

We now need to calculate the size of $G$. 
The number of $r$-sets meeting exactly $t$ of the $K_k^{(r)}$s is $O\left(n^t\right)$. Note that $G$ does not contain any edges intersecting more than $r-2$ of the $K_k^{(r)}$s and thus we have
$e(G) = O\left(n^{r-2}\right)$ and Claim \ref{claim2:when k|n} follows.
\end{proof}

Now we will consider the case where $k$ does not divide $n$.

\begin{claim} $\sat = \Omega\left(n^{r-1}\right)$ when $k \nmid n$.
\end{claim}
\begin{proof}

let $G = (V,E)$ be a graph on $n$ vertices that is $\family$-saturated. 
Let $A$ be the set of all vertices of $G$ that are contained in a $K_k^{(r)}$, and $B = V\setminus A$ be all vertices not contained in any $K_k^{(r)}$. 

The choice of family $\family$ implies that all of the copies of $K_k^{(r)}$ contained in $A$ must be disjoint:
\begin{itemize}
\item $F$ forbids two $K_k^{(r)}$s intersecting in exactly one vertex.
\item $I_2$ forbids two $K_k^{(r)}$s intersecting in at least two vertices and each containing at least $r-2$ vertices not in the intersection.
\item $I_{r-1}$ forbids two $K_k^{(r)}$s intersecting in at least $r-1$ vertices and each containing at least one vertex not in the intersection.
\end{itemize}
Since $k$ was chosen with $k \ge 2r-4$, any two intersecting copies of $K_k^{(r)}$s fall in one of these three categories.

Thus $A$ consists of $m$ disjoint copies of $K_k^{(r)}$ for some $m$, together with some extra edges that go between them. Since $k$ does not divide $n$, we can  conclude that $A$ is not all of $V(G)$, or equivalently that $B$ is non-empty. 

We make the following two claims about the number of edges in $G$:

\begin{claim} 
If $m \ge r-1$ then $G$ contains at least $\binom{m}{r-1}k^{r-1}$ edges consisting of $r-1$ vertices in $A$ and one vertex in $B$.
 \label{claim2:edges between A and B}
\end{claim}

\begin{claim}
If $|B| \ge k-1$, then $G$ contains $\binom{|B|}{r-1}\frac{k-r}{r}$ edges.
\label{claim2:degrees in B}
\end{claim}

We can use these two claims to deduce the result. 
One of $A$ and $B$ contains at least half the vertices in $G$. 
If $|A| = mk \ge \frac{n}{2}$, then using claim \ref{claim2:edges between A and B} the number of edges in $G$ is at least $\Omega\left(n^{r-1}\right)$.
If $|B| \ge \frac{n}{2}$, then using claim \ref{claim2:degrees in B} the number of edges in $G$ is at least $\Omega\left(n^{r-1}\right)$.

All that is left is to prove the two claims.

\begin{proof}[Proof of Claim \ref{claim2:edges between A and B}.]
Fix $r-1$ vertices $v_1, \ldots, v_{r-1}$, each in a different copy of $K_k^{(r)}$ in $A$. 
We will show that $G$ contains at least one edge containing all of $v_1, \ldots v_{r-1}$ together with a vertex in $B$.

If all possible such edges exist, then we are done (recall that $B$ is non-empty).

Otherwise, there is some $x$ in $B$ such that the $r$-set $\{x,v_1,\ldots,v_{r-1}\}$ is not in $E(G)$.
Adding this edge must create a graph in $\family$, and so it must create a new copy of $K_k^{(r)}$ (it cannot be any of the `extra' edges in $I_2$, $I_{r-1}$ or $H$).

Consider this new $K_k^{(r)}$. 
Suppose for a contradiction it contains some vertex $y$ in ${A \setminus \{ v_1 \ldots v_{r-1}\}}$. 
If $y$ is in the same original $K_k^{(r)}$ as one of the $v_i$s then the edge $\{y,v_1,\ldots v_{r-1}\}$ creates a copy of $I_2$. 
If $y$ is not in the same original $K_k^{(r)}$ as any of the $v_i$s then the edge $\{y,v_1,\ldots v_{r-1}\}$ creates a copy of $H$. 
Thus $\{y,v_1,\ldots v_{r-1}\}$ is not in $E(G)$, and we have a contradiction.

Thus the other vertices of this new $K_k^{(r)}$ are in $B$, and $G$ contains $k-r > 1$ edges containing all of $v_1, \ldots v_{r-1}$ together with a vertex in $B$.
\end{proof}

\begin{proof}[Proof of Claim \ref{claim2:degrees in B}.]
Fix a set $X$ of $r-1$ vertices in $B$. We will show that $X$ is contained in at least $k-r$ edges in $G$.

Suppose first that $X \cup \{y\}$ is an edge for all $y$ in $B \setminus X$. Then $X$ is in $|B| - |X| \ge k-r$ edges as required.

Otherwise, there exists some $y$ in $B\setminus X$ such that $X \cup \{y\}$ is not in $E(G)$. 
Note that adding the edge $X \cup \{y\}$ must create some graph in $\family$. It must thus create a $K_k^{(r)}$, since $B$ contains no $K_k^{(r)}$s and so it cannot be the `extra' edge in $I_2$, $I_{r-1}$ or $H$. 
Then we have that $X$ is contained in $k-r$ other edges in that $K_k^{(r)}$.

Thus every $r-1$ set in $B$ is contained in at least $k-r$ edges. Each edge in $G$ contains at most $r$ different $(r-1)$-sets in $B$, and so the total number of edges in $G$ is at least
\[ \binom{|B|}{r-1} \frac{k-r}{r}. \]
\end{proof}
\end{proof}
\end{proof}

\section{Obtaining an Small Saturation Number on a Denser Set}
\label{sec:dense}
In both Theorem \ref{thm:big family} and Theorem \ref{thm:constant size family}, the saturation number is asymptotically small ($O\left(n^{r-2}\right)$) when $n$ is divisible by $k$  and asymptotically large ($\Theta \left(n^{r-1}\right)$) for all other values of $n$. Since $k$ is at least as big as $r$, this means that the set of values where the saturation number is asymptotically small has low density -- less than $1/r$.
It is natural to ask whether it is possible to have a forbidden family where the saturation number has different asymptotics for complementary subsets of the naturals of equal density. For example, could we ensure $\sat$ is asymptotically small on even numbers and asymptotically large on odd numbers? 

It turns out that it is possible. The proof will use a family similar to the one in Theorem \ref{thm:big family}. However, rather than just using $K_k^{(r)}$ as a base graph, we will take two base graphs of different even orders. The family will contain all possible intersections of the two base graphs, and all graphs consisting of disjoint unions of copies of the base graphs together with a bridge edge. 

For large even $n$ there is an $\family$-saturated graph on $n$ vertices that uses few edges; namely taking disjoint copies of the base graphs. However, for odd $n$ we will need to use many more edges.

A first attempt at choosing the two base graphs might be to take $K_k^{(r)}$ and $K_{k+2}^{(r)}$ for some even $k$. However, $K_{k+2}^{(r)}$ contains two $K_k^{(r)}$s intersecting in $k-2$ vertices, which is a graph we would want to include in our forbidden family. This is a problem, as if that graph was forbidden, all copies of $K_{k+2}^{(r)}$ would be forbidden too. 

Instead, we will take $K_k^{(r)}$ and any graph on $k+2$ vertices which has certain helpful properties: one of which is that it contains only one copy of $K_k^{(r)}$.

\begin{theorem} \label{thm:evens and odds}
For all $r \ge 2$ there exists a family $\family$ of $r$-graphs such that 
	$$ \sat =
		\begin{cases}
			O(n) & \text{if $n$ is even}\\			
			\Omega\left(n^{r-1}\right) & \text{if $n$ is odd.} 
		\end{cases} 
	$$
\end{theorem}

\begin{proof}

Fix any even integer $k>r+1$. Let $K = K_k^{(r)}$ and let $L$ be a graph which satisfies the following properties:
\begin{itemize}
\item $L$ has $k+2$ vertices;
\item $L$ is connected;
\item $L$ contains exactly one copy of $K_k^{(r)}$; and
\item For any edge $e$ of $L$ and any $(r-1)$-sized subset $s$ of the edge $e$, there is another edge $e'$ in $L$ such that $e \cap e' = s$.
 \end{itemize}
One such $L$ consists of a $K_k^{(r)}$ and two $K_{k-1}^{(r)}$s with a common intersection of $k-2$ vertices. It is easy to see that this has the required properties.

We call $k$ and $L$ the \emph{base graphs}. An example of the these two graphs for $k=18$ can be seen in figure \ref{fig:evens}, where if a set of vertices is surrounded by a dashed line then all edges contained in that set exist.
\begin{figure}[!h]
  \centering
    \includegraphics[width=.6\textwidth]{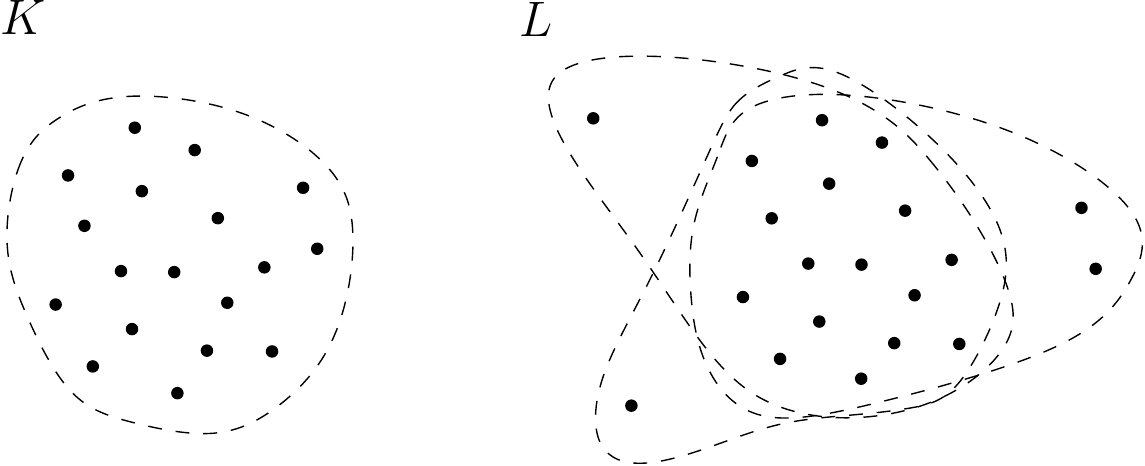}
    \caption{The two graphs under consideration for $k=14$.}
    \label{fig:evens}
\end{figure}

Take $\family$ to be the family containing all of the following $r$-graphs:
\begin{enumerate}[a)]
\item Every graph comprising $t$ disjoint graphs $H_1,H_2,\ldots H_t$ (for $2 \le t \le r$) where each $H_i$ is a base graph, together with an edge $E$ meeting each $H_i$ in $x_i \ge 1$ vertices such that $\sum_{i=1}^t x_i = r$. We call $E$ a bridge edge.
\item Every graph comprising two base graphs on vertex sets $V_1$ and $V_2$ with non-empty intersection and neither contained in the other -- that is, $V_1 \cap V_2, V_1 \setminus V_2$ and $V_2 \setminus V_1$ all non-empty.
\item Every graph comprising $L$ plus a single extra edge on the same vertex set.
\end{enumerate}

First, let us deal with the case where $n$ is even. For all $n$ sufficiently large, (in particular, at least $\frac{k(k-2)}{2}$), we can write $n$ as a sum $ak + b(k+2)$ for some $a,b \in \mathbb{N}$. 

Take $G$ to be a graph on $n$ vertices consisting of $a$ disjoint copies of $K$ and $b$ disjoint copies of $L$. It is clear that adding any edge will create a graph in the family $\family$: adding a missing edge between base graphs creates a graph of type (a), and adding a missing edge within a copy of $L$ creates a graph of type (c).
Thus 
\begin{align*}
\sat \le \edge(G) &\le  a\binom{k}{r} + b \binom{k+2}{r} \\
&\le (ak + b(k+2))\frac{1}{r}\binom{k+1}{r-1}\\
& = O(n)
\end{align*}

Now suppose that $n$ is odd and let $G = (V,E)$ be a graph on $n$ vertices that is $\family$-saturated.

Let $A$ be the set of all vertices of $G$ that are contained in a copy of one of the base graphs, and $B = V\setminus A$ be all vertices not contained in any copy of a base graph. 

Note that the subgraph induced on $A$ must consist of disjoint copies of the two base graphs, by the choice of family $\family$. This implies that $B$ is not empty, since $n$ is odd and both base graphs have an even number of vertices.

Let $X$ be an $r$-set meeting $B$ that is not in $E(G)$. Adding the edge $X$ to $G$ must create some graph in $\family$. The edge $X$ cannot form a bridge between two $K_k^{(r)}$s by definition of $B$, and it also cannot add an extra edge to an existing copy of $L$ for the same reason. Thus adding such an edge must create a copy of one of the base graphs, $K$ or $L$.

We make the following two claims about the number of edges in $G$:

\begin{claim} $G$ contains at least \[\binom{\frac{|A|}{k+2}}{r-1}\] edges consisting of $r-1$ vertices in $A$ and one vertex in $B$.
 \label{claim3:edges between A and B}
\end{claim}

\begin{claim}
If $|B| \ge k-1$ then $G$ contains at least $\binom{|B|}{r-1} \frac{k-r}{r}$ edges.
\label{claim3:degrees in B}
\end{claim}

We can use these two claims to deduce the result. 
One of $A$ and $B$ contains at least half the vertices in $G$.
If $|A| \ge \frac{n}{2}$, then using claim \ref{claim3:edges between A and B} the number of edges in $G$ is at least $\Omega\left(n^{r-1}\right)$. 
If $|B| \ge \frac{n}{2}$, then using claim \ref{claim3:degrees in B} the number of edges in $G$ is $\Omega\left(n^{r-1}\right)$.

All that is left is to prove the two claims.

\begin{proof}[Proof of Claim \ref{claim3:edges between A and B}.]
Fix $r-1$ vertices $v_1, \ldots, v_{r-1}$, each in a different base graph in $A$ (of which there are at least $\frac{|A|}{k+2}$). 
We will show that $G$ contains at least one edge containing all of $v_1, \ldots v_{r-1}$ together with a vertex in $B$. Then the number of edges between $A$ and $B$ is at least the desired amount.

If all possible such edges exist, then we are done (recall that $B$ is non-empty).

Otherwise, there is some $x$ in $B$ such that the $r$-set $\{x,v_1,\ldots,v_{r-1}\}$ is a missing edge.
Adding this edge must create one of the graphs in $\family$. It cannot be a bridge edge as $B$ contains no copies of $L$. It also cannot be the extra edge in a copy of `$L$ plus an edge', as no vertex in $B$ is contained in a copy of $L$. Thus adding the edge must create a new copy of one of the base graphs, $K_k^{(r)}$ or $L$. 

The other vertices of this new base graph cannot be in $A$: if $y$ is in $A \setminus \{ v_1 \ldots v_{r-1}\}$, then $\{y,v_1,\ldots v_{r-1}\}$ is a non-edge, otherwise it serves as a bridge edge and $G$ contains a graph of type (a) in the family $\family$.

So the other vertices of this new base graph are all in $B$. 
We then have that there is at least one edge containing all of $v_1, \ldots v_{r-1}$ together with a vertex in $B$: this is obviously true if the base graph was $K_k^{(r)}$, and true by the properties insisted upon earlier if the base graph was $L$.
\end{proof}

\begin{proof}[Proof of Claim \ref{claim3:degrees in B}.]
To apply a similar proof to before, we first want to show that if $Y$ is an $r$-set in $B$ that is not in $E(G)$, then adding $Y$ to $G$ creates a new copy of $K_k^{(r)}$. Suppose for contradiction this is not the case.

Adding $Y$ must create a graph in the family $\family$, so $Y$ must create one of:
\begin{itemize}
\item a graph of type (a) in $\family$;
\item a copy of the graph $L$; or
\item a copy of the graph $L$ plus an edge.
\end{itemize} 
However, no vertex in $B$ is in a copy of one of the base graphs. This implies both that $Y$ cannot be a bridge edge between copies of the base graph and also that $Y$ cannot be an extra edge added to copy of $L$. 
Thus we must have that $Y$ creates a copy of $L$.

Note that $L$ contains a copy of $K_k^{(r)}$. Since $Y$ does not create a $K_k^{(r)}$, this  $K_k^{(r)}$ must already exist. However, then $Y$ must intersect this $K_k^{(r)}$ in $r-2$ vertices, contradicting that no vertex in $B$ is contained within a copy of $K_k^{(r)}$. 

Now, fix a set $X$ of $r-1$ vertices in $B$. We will show that $X$ is contained in at least $k-r$ edges in $G$.

Suppose first that $X \cup \{y\}$ is an edge for all $y$ in $B \setminus X$. Then $X$ is in $|B| - |X| \ge k-r$ edges as required.

Otherwise, there exists some $y$ in $B\setminus X$ such that $X \cup \{y\}$ is not in $E(G)$. 
Note that adding the edge $X \cup \{y\}$ must create some graph in $\family$. It must thus create a $K_k^{(r)}$, since $B$ contains no $K_k^{(r)}$s and so it cannot be the `extra' edge in $I_2$, $I_{r-1}$ or $H$. 
Then we have that $X$ is contained in $k-r$ other edges in that $K_k^{(r)}$.

Thus every $r-1$ set in $B$ is contained in at least $k-r$ edges. Each edge in $G$ contains at most $r$ different $(r-1)$-sets in $B$, and so the total number of edges in $G$ is at least
\[ \binom{|B|}{r-1} \frac{k-r}{r}. \]

\end{proof}
\end{proof}

\section{Some Open Problems}
\label{sec:problems} 

The main questions still left unanswered are the two conjectures in Section 1; that is, Tuza's conjecture and its generalisation to $r$-graphs. There are some subsidiary questions that may help with progress towards that goal: the first towards a counterexample and the second towards a proof. 

In Theorem \ref{thm:constant size family}, we defined $H$ to be the graph consisting of $r$ disjoint copies of $K_k^{(r)}$ and an edge intersecting each $K_k^{(r)}$. This seems somehow to be the key graph in ensuring a nice construction when $k$ divides $n$. Perhaps, then, this would be a good target for a counterexample to Tuza's conjecture: 
\begin{question}
If $H$ is as above, does $\frac{\operatorname{Sat}(H,n)}{n^{r-1}}$ tend to a limit as $n$ tends to infinity? \label{question:H}
\end{question}

In the case $r=2$, the answer to question \ref{question:H} is affirmative. 
Write  $n = mk + c$ where $0 \le c<k$ and let $G$ be the graph  on $n$ vertices consisting of $m-1$ disjoint $K_k$s and one $K_{k+c}$. This graph $G$ is certainly $H$-saturated and has $\frac{k-1}{2}n + O(1)$ edges. One can show that in a minimal $H$-saturated graph every vertex must have degree at least $k-1$. The argument is similar to one by Pikhurko (example 4, \cite{Pik04}); we omit the details. These together imply that $\operatorname{Sat}(H,n) = \left(\frac{k-1}{2} + o(1)\right)n$.

For $r>2$, however, a similar construction does not easily work. Let $G$ be a collection of disjoint $K_k^{(r)}$s. Any $r$-set not in $E(G)$ that meets some $K_k^{(r)}$ in two or more vertices could be added to $G$ without creating a copy of $H$. However, if we add several such edges to $G$ then we possibly create other copies of $K_k^{(r)}$ intersecting the original ones, making it difficult to be sure there are no copies of $H$. Thus it is not clear what a minimal saturated graph would look like even when $n$ is divisible by $k$.

In going from a forbidden family of size that grows with $r$ in Theorem \ref{thm:big family} to a family of constant size in Theorem \ref{thm:constant size family}, we lost the large gap in the asymptotics for $\sat$. That is, in the case when $n$ is divisible by $k$, the construction of a saturated graph with few edges went from having $\Theta(n)$ edges to having $\Theta\left(n^{r-2}\right)$ edges. Is it possible to retain the large difference in asymptotics and still decrease the size of the family? This seems difficult, especially if we try to reduce the family to a single graph.
\begin{question}
Let $F$ be an $r$-graph. Can $\operatorname{Sat}(F,n)$ be $O(n)$ for some infinite sequence of values of $n$ and $\Omega(n^{r-1})$ for some other infinite sequence? \label{question:gap}
\end{question}

If Tuza's conjecture is true, it would imply that the answer to question \ref{question:gap} is `no'. However, it may be easier to provide a negative answer to question \ref{question:gap} than to prove Tuza's conjecture directly, and an answer might help provide ideas towards a full proof.

\nocite{*}

\bibliographystyle{alpha}
\bibliography{bib} 

\begin{thebibliography}{FFS11}

\bibitem[Bol65]{Bol65}
B\'ela Bollob\'as.
\newblock On generalised graphs.
\newblock {\em Acta Math. Acad. Sci. Hung.}, 16:447--452, 1965.

\bibitem[Bol86]{Bol86}
B\'ela Bollob\'as.
\newblock {\em Combinatorics}.
\newblock Cambridge University Press, Cambridge, 1986.
\newblock Set systems, hypergraphs, families of vectors and combinatorial
  probability.

\bibitem[FFS11]{FFS11}
Jill~R. Faudree, Ralph~J. Faudree, and John~R. Schmitt.
\newblock A survey of minimum saturated graphs.
\newblock {\em Electron. J. Combin.}, 18, 2011.
\newblock Dynamic Survey 19.

\bibitem[KT86]{KT86}
L.~K\'aszonyi and Zs. Tuza.
\newblock Saturated graphs with minimal number of edges.
\newblock {\em J. Graph Theory}, 10(2):203--210, 1986.

\bibitem[Pik99]{Pik99}
Oleg Pikhurko.
\newblock The minimum size of saturated hypergraphs.
\newblock {\em Combin. Probab. Comput.}, 8(5):483--492, 1999.

\bibitem[Pik04]{Pik04}
Oleg Pikhurko.
\newblock Results and open problems on minimum saturated hypergraphs.
\newblock {\em Ars Combin.}, 72:111--127, 2004.

\bibitem[Tuz88]{Tuz88}
Zsolt Tuza.
\newblock Extremal problems on saturated graphs and hypergraphs.
\newblock {\em Ars Combin.}, 25(B):105--113, 1988.
\newblock Eleventh British Combinatorial Conference (London, 1987).

\end{thebibliography}

\end{document}